\documentclass[12pt]{article}
\usepackage{amsmath,amssymb,bbm,subfig,url,ulem}
\usepackage{amsthm}
\usepackage{xcolor}
\usepackage[T1]{fontenc}
\usepackage[margin=1in]{geometry}

\usepackage{tikz}
\usepackage[bookmarks=false,unicode,colorlinks,urlcolor=red,citecolor=red,linkcolor=blue]{hyperref}

\usepackage{subfig}
\usetikzlibrary{calc}
\usetikzlibrary{arrows}
\usetikzlibrary{patterns}
\usetikzlibrary{through}
\usetikzlibrary{plotmarks}

\numberwithin{equation}{section}

\usepackage{aliascnt}
\newtheorem{theorem}{Theorem}[section]

\newaliascnt{lemma}{theorem}
\newtheorem{lemma}[lemma]{Lemma}
\aliascntresetthe{lemma}

\newaliascnt{proposition}{theorem}
\newtheorem{proposition}[proposition]{Proposition}
\aliascntresetthe{proposition}

\newaliascnt{corollary}{theorem}
\newtheorem{corollary}[corollary]{Corollary}
\aliascntresetthe{corollary}

\newaliascnt{conjecture}{theorem}

\aliascntresetthe{conjecture}

\newaliascnt{remark}{theorem}

\aliascntresetthe{remark}

\newaliascnt{definition}{theorem}

\aliascntresetthe{definition}

\makeatletter
\def\tagform@#1{\maketag@@@{\ignorespaces#1\unskip\@@italiccorr}}
\let\orgtheequation\theequation
\def\theequation{(\orgtheequation)}
\makeatother
\def\equationautorefname~{}

\newcommand{\mref}[1]{%
\href{http://www.ams.org/mathscinet-getitem?mr=#1}{#1}}

\newcommand{\arxiv}[1]{%
\href{http://front.math.ucdavis.edu/#1}{ArXiv:#1}}

\newcommand{\Real}{{\mathbb{R}}}
\newcommand{\R}{\mathbb{R}}

\newcommand{\Rd}{{\Real^d}}
\newcommand{\Rdz}{{\Real^d_0}}
\newcommand{\Rdd}{{\Real^{d-1}}}

\newcommand{\Sd}{{\mathbb{S}^{d-1}}}
\newcommand{\Sp}{\mathcal{S}}
\newcommand{\Sdd}{{\mathbb{S}^{d-2}}}

\newcommand{\lap}{{\Delta^{\alpha/2}}}

\newcommand{\lapS}{{\Delta^{\alpha/2}_\Sd}}
\newcommand{\one}{{\bf1}}
\newcommand{\cda}{C_{d-1,\alpha}}
\newcommand{\Cda}{C_{d,\alpha}}
\newcommand{\Cdmja}{C_{d-1,\alpha}}

\newcommand{\Ad}{{\mathcal{A}_{d,\alpha}}}

\newcommand{\limeps}{\lim_{\epsilon\to 0}}

\newcommand{\ps}[2]{#1{\cdot}#2}
\DeclareMathOperator{\dist}{dist}

\newcommand{\BigO}{O}

\newcommand{\ty}{{\widetilde{y}}}
\newcommand{\tx}{{\widetilde{x}}}
\newcommand{\tys}{{\widetilde{y^*}}}
\newcommand{\tz}{{\widetilde{z}}}
\newcommand{\tzs}{{\widetilde{z^*}}}
\newcommand{\ve}{{\varepsilon}}

\definecolor{bs}{RGB}{128,0,0}
\definecolor{kb}{RGB}{0,128,0}
\definecolor{as}{RGB}{150,150,10}

\begin{document}

\title{Martin kernel for fractional Laplacian in narrow cones\footnote{MSC2010: 31B25, 45K05 (primary); 60J45 (secondary). Key words and phrases: fractional Laplacian, Martin kernel, circular cone.}}
\author{Krzysztof Bogdan\footnote{Institute of Mathematics and Computer Science, Wroc{\l}aw University of Technology, Poland, email: bogdan@pwr.wroc.pl, this author was partially supported by NCN grant 2012/07/B/ST1/03356}, Bart{\l}omiej Siudeja\footnote{Department of Mathematics, University of Oregon, Eugene, OR 97403, USA, email: siudeja@uoregon.edu, this author was partially supported by NCN grant 2012/07/B/ST1/03356}, Andrzej St\'os\footnote{Laboratoire de Math\'ematiques, Clermont Universit\'e,
Universit\'e Blaise Pascal and CNRS UMR 6620,
BP 80026, 63171 Aubi\`ere, France, email: stos@math.univ-bpclermont.fr}
}

\maketitle
\abstract{
We give a power law for
the homogeneity degree of the Martin kernel of the fractional Laplacian for the right circular cone 
when the angle of the cone tends to zero.}

\section{Introduction and main result}

For $d\geq 2$ and $0<\Theta< \pi$, we consider the right circular cone of angle $\Theta$ (aperture $2\Theta$):
\begin{gather}\label{gtheta1}
\Gamma_\Theta=\left\{ x=(x_1,\ldots,x_d)\in\Rd:\;x_d>|x|\cos\Theta \right\}.
\end{gather}
The Martin kernel of the fractional Laplacian $\Delta^{\alpha/2}$, $0<\alpha<2$, for $\Gamma_\Theta$
is the unique continuous function $M\geq 0$ on $\Rd$, such that $M$ is smooth on $\Gamma_\Theta$, $\Delta^{\alpha/2}M=0$ on $\Gamma_\Theta$,
$M=0$ on $\Gamma_\Theta^c$, and $M(1,0,\ldots,0)=1$.
It is known that $M$ is $\beta$-homogeneous:
  \begin{gather*}
    M(x)=|x|^\beta M(x/|x|),\qquad x\in \Rdz,
  \end{gather*}
where $\Rdz=\Rd\setminus \{0\}$ and $\beta=\beta(d, \alpha,\Theta)\in (0,\alpha)$. For instance, $\beta=\alpha/2$ for the half-space ($\Theta=\pi/2$). These facts are given in \cite[Theorem 3.2]{MR2075671}, see also \cite[Theorem 3.9]{MR1671973}. 
The homogeneity degree $\beta$ is crucial 
for precise asymptotics of nonnegative harmonic functions of $\Delta^{\alpha/2}$ in cones.
In fact, the Martin, Green and heat kernels of $\Delta^{\alpha/2}$ for $\Gamma_\Theta$ 
enjoy explicit elementary estimates in terms of $\beta$ 
\cite[Lemma 3.3]{MR2213639}, \cite[(23)]{MR2602155}.
By domain monotonicity of the Green function
and by the boundary Harnack principle \cite{MR2365478}, the knowledge of $\beta$ has  important consequences for the potential theory of $\Delta^{\alpha/2}$ with Dirichlet boundary conditions
in 
general Lipschitz domains, see, e.g. \cite[Theorem~5.2]{MR2114264}.
Furthermore,
$\beta$
determines the critical moment $p_0=\beta/\alpha$ of integrability of the first exit time of the isotropic 
$\alpha$-stable L\'evy process $\{X_t,\,t\geq 0\}$ in $\Rd$ from $\Gamma_\Theta$ \cite[Lemma 4.2]{MR2075671}, which
is a long-standing motivation to study $\beta$, cf. \cite{MR0474525}, \cite{MR1750907}, \cite{MR2075671}, \cite{MR1062058}, \cite{MR1936081}.
The connections of $\Delta^{\alpha/2}$ to the isotropic $\alpha$-stable L\'evy process in $\Rd$ are well-known and can be found in the references; below we focus on analytic construction of superharmonic functions of $\Delta^{\alpha/2}$ in $\Gamma_\Theta$.

From \cite{MR2213639}, \cite{MR2075671} and \cite{MR1750907}, $\beta=\beta(d, \alpha,\Theta)$ is strictly decreasing in $\Theta$, and $\beta\to\alpha$ as $\Theta\to 0$. 
This contrasts 
with the
case of 
the classical Laplacian, e.g.
the Martin kernel for  the Laplacian and planar sector with aperture $2\Theta\in (0,2\pi)$ has homogeneity degree equal to $\pi/(2\Theta)$, which is arbitrarily large for narrow enough cones (see \cite{MR0474525} for higher dimensions). 
The problem of giving a more quantitative  description of $\beta$ remained a puzzle for
over a decade, since \cite{MR2075671, MR1750907}.
In this work we prove a power law for $\beta(d, \alpha,\Theta)$ as $\Theta\to 0$. 
Namely, let $0<\alpha<2$, 
\begin{align}\label{omega}
    \omega(\Theta)&=\begin{cases}
      \Theta^\alpha& \mbox{ if } 0<\alpha<1, \\
      \Theta |\log \Theta| & \mbox{ if } \alpha=1,\\
      \Theta& \mbox{ if } \alpha>1,
    \end{cases}
  \end{align}
and
\begin{align}\label{eq:defs} 
B_{d,\alpha}
  &= 
\frac{\Gamma\left(\frac{d+\alpha}{2} \right)}{\pi^{3/2}\Gamma\left(\frac{d-1+\alpha}{2} \right)}\sin\left( \frac{\pi\alpha}{2} \right)B\left(1+\frac{\alpha}{2},\frac{d-1}{2}\right),
\end{align}
where $\Gamma$ and $B$ are the Euler gamma and beta functions, respectively.
For asymptotic results,  we shall often use Landau's $\BigO$ notation. Here is our main theorem.
\begin{theorem}\label{mainthm}
If $\Theta\to 0$, then
 $
\beta(d,\alpha,\Theta)=\alpha-
\Theta^{d-1+\alpha}\big[
B_{d,\alpha}
+\BigO(\omega(\Theta))\big].
$
\end{theorem}
The result is proved in Section~\ref{sec:pmr} below.
The exponent 
$d-1+\alpha$
was conjectured by Tadeusz Kulczycki in a private conversation on the methods of
\cite{MR1750907}. 
Here is an overview of our development. We consider the unit sphere: 
$$\Sd=\{x\in \R^d: \, |x|=1\},$$
and
the spherical cap of $\Gamma_\Theta$:
\begin{gather}\label{btheta1}
  B_\Theta=\left\{ \theta\in\Sd :\; \theta_d> \cos\Theta \right\}.
\end{gather}
If $\phi$ is a function on $\Sd$, $\Phi(x)=|x|^{\lambda}\phi(|x|^{-1}x)$, $x\in \Rdz$, and $\lambda\in (-d,\alpha)$, then we have the following
decomposition:
$$\Delta^{\alpha/2}\Phi(x)=\Delta^{\alpha/2}_\Sd\phi(x)+R_{\lambda}\phi(x), \qquad x\in \Sd.$$ 
The spherical part $\Delta^{\alpha/2}_\Sd$ is an integro-differential operator on $\Sd$ akin to the fractional Laplacian in dimension $d-1$. The radial part $R_\lambda$ is an integral operator on $\Sd$ whose nonnegative kernel increases in $\lambda\in((\alpha-d)/2,\alpha)$, in fact explodes as $\lambda \to \alpha$. 
We have $\Delta^{\alpha/2}_\Sd M=-R_\beta M$ on $B_\Theta$. Heuristically,  $\beta$ is a generalized  eigenvalue of $\Delta^{\alpha/2}_\Sd$ with Dirichlet conditions, relative
to the family $R_\lambda$.
In the classical case $\alpha=2$, the operator $R_\gamma$ reduces to multiplication by $\gamma$, which leads to a genuine eigenproblem on the sphere (see e.g. \cite{zbMATH06135849}).
To estimate $\beta$ we define a suitable spherical profile function $\phi$ on $\Sd$ supported on $B_\Theta$, extend it  to be $\lambda$-homogeneous on $\Rd$ and choose $\lambda$ so that the extension is either superharmonic or subharmonic for $\Delta^{\alpha/2}$. This yields lower and upper bounds for 
$M$ and $\beta$ by means of the maximum principle for
$\Delta^{\alpha/2}$. Namely, $\Delta^{\alpha/2}_\Sd\phi$ is provided by a judicious choice of $\phi$, and 
we control $\alpha-\beta$ by proving uniform estimates for the kernel of $R_\lambda$.

Two candidates offer themselves to construct $\phi$:
the principal eigenfunction of the fractional Laplacian for the ball in dimension $d-1$ and the expected exit time of the isotropic $\alpha$-stable L\'evy process from the  ball. 
Surprisingly, it is the latter choice that allows us to handle the super- and subharmonicity of $\Phi$ 
for $\Delta^{\alpha/2}$ up to the boundary of $\Gamma_\Theta$. 
The expected exit time of the ball has the additional advantage 
of being explicit, allowing us to construct explicit barriers (i.e. superharmonic functions vanishing at the boundary) for narrow cones.
We remark that the principal eigenfunction and eigenvalue of the ball for $\Delta^{\alpha/2}$ are not known (see \cite{MR2974318} for bounds and references),
therefore the above expression for 
$B_{d,\alpha}$
 is a remarkable serendipity. We note in passing the Martin kernel of $\Gamma_\Theta$ with the pole at the origin  is $|x|^{\alpha-d}M(x/|x|^2)$, $x\in \Rdz$ \cite{MR2075671}, hence its singularity at the origin is roughly $|x|^{-d+B_{d,\alpha} \Theta^{d-1+\alpha}}$ for small $\Theta$.
This exemplifies some of the extreme behaviour of nonnegative $\alpha$-harmonic functions at the boundary of (narrow) Lipschitz open sets.

The structure of the paper is as follows. The main line of arguments is presented in \autoref{sec:prel}, where 
we give preliminaries, detail  the above decompositions of $\Delta^{\alpha/2}$ and state precise asymptotic results for the kernels of the spherical and radial operators. As mentioned, the spherical profile $\phi$ is constructed from
the expected exit time of the ball for the isotropic $\alpha$-stable L\'evy process in $\R^{d-1}$. We also use 
a suitable variant of 
Kelvin transform to define $\phi$. We then estimate $\Delta^{\alpha/2}_\Sd\phi+R_\lambda\phi$.
The proof of \autoref{mainthm} is given at the end of \autoref{sec:prel}.
In \autoref{sec:ug} we collect the more technical proofs from \autoref{sec:prel} and some auxilary results. We also prove the following result
for the classical Laplacian in the complement of a plane slit by a cone, using
\autoref{mainthm} and the connection of $\Delta$ and 
$\Delta^{1/2}$ in dimensions $d$ and $d-1$, respectively.
\begin{corollary}\label{cor:slit}
Let $V=\{(x_1,x_2,x_3)\in \R^3: x_3\neq 0 \mbox{ or } x_2>\sqrt{x_1^2+x_2^2}\cos \Theta\}$. The homogeneity exponent of Martin kernel of $\Delta$ and $V$ 
is $1-\Theta^2/4+\BigO(\Theta^3\log \Theta)$ as $\Theta\to 0^+$.
\end{corollary}

Our work leads to interesting new problems.
It is worthwhile to 
study the above generalized eigenproblem of $\Delta^{\alpha/2}_\Sd$  in the setting of $L^2(\Sd,\sigma)$, as opposed to the present pointwise setting. Some results in this direction are given in \cite{MR2075671}.  It would be very interesting to understand the generalized higher-order eigenfunctions of $\Delta^{\alpha/2}_\Sd$ with respect to $R_\lambda$ and, ultimately, {oscillating} harmonic functions of 
$\Delta^{\alpha/2}$.
We note that {nonnegative} harmonic functions of $\Delta^{\alpha/2}$ have been completely described in \cite{MR2365478}, but oscillating harmonic functions of the operator are hardly understood. Similar problems are relevant for more general nonlocal L\'evy-type
operators, of which 
$\Delta^{\alpha/2}$
is but a prominent example. For instance, the homogeneity of the Martin kernel of cones for more 
general stable L\'evy processes should now be available by using the methods of \cite{MR2075671} and the boundary Harnack principle recently proved in \cite{2012arXiv1207.3160B}.

\noindent
{\bf Acknowledgements:}
We thank Tadeusz Kulczycki for discussions and suggestions on \autoref{mainthm} and \autoref{cor:slit}.
We thank Amir Dembo for stimulating discussions on Martin kernels. Krzysztof Bogdan thanks
the Department of Statistics at Stanford University and the Institute of Mathematics of the Polish Academy of Sciences for their hospitality during his work on the paper.

\section{Homogeneous superharmonic functions}\label{sec:prel}
\subsection{Decomposition of the fractional Laplacian}
Below we let $d\ge2$ and $0<\alpha<2$, unless explicitly stated otherwise. (The reader may consult \cite{MR2075671} for $d=1$ and \cite{MR0474525} for  $\alpha=2$.) 
All sets, functions and measures on $\Rd$ considered below are
assumed Borel. By $dx$ we denote the Lebesgue measure on $\Rd$ or $\Rdd$, depending on context.
We let $\sigma$ be the $(d-1)$-dimensional Hausdorff measure (surface measure), so normalized that
$$
\omega_d:=\sigma(\Sd)=2\pi^{d/2}/\Gamma(d/2).
$$
Let $0<\Theta<\pi$ and $\one=(0,0,\dots,0,1)$. We have
\begin{align*}
\Gamma_\Theta=\left\{ x\in\Rd:\; { \ps{x}{\one} >|x|\cos\Theta} \right\} \quad \mbox{ and } \quad
B_\Theta=\left\{ \theta\in\Sd :\; \ps\theta\one > \cos\Theta \right\}.
\end{align*}
Let
\[ \Ad = \alpha 2^{\alpha-1} \pi^{-d/2}  \Gamma\left(\frac{d+\alpha}{2}\right) / \Gamma\left(1-\frac{\alpha}{2}\right).\] 
For a real-valued function $\Phi$ on $\Rd$, twice continuously differentiable, i.e. $C^2$
near some point $x\in \Rd$ and such that
\begin{align}\label{eq:dc}
&\int_\Rd |\Phi(y)|(1+|y|)^{-d-\alpha}dy<\infty,
\end{align}
we define
\begin{equation}\label{eq;deful}
  \Delta^{\alpha/2}\Phi(x) =  
  \lim_{\epsilon \downarrow 0}\Ad\int_{|y{-x}|>\epsilon}
  \left[\Phi(y)-\Phi(x)\right]|y-x|^{-d-\alpha}dy,
\end{equation}
the fractional Laplacian of $\Phi$ at $x$, and
we say $\Phi$ is $\alpha$-harmonic, i.e. harmonic for $\Delta^{\alpha/2}$, on open $D\subset \Rd$ if 
$\Delta^{\alpha/2}\Phi(x)=0$ for $x\in D$
(see \cite{MR1671973, MR1825645, MR2569321} for broader discussion).
If $x\not \in {\rm supp}\, \Phi$, then
\begin{equation}\label{eq;defulpn}
  \Delta^{\alpha/2}\Phi(x) = 
  \Ad \int_{\Rd}
  \Phi(y)  |y-x|^{-d-\alpha} dy\,.
\end{equation}
If $r>0$ and $\Phi_r(x)=\Phi(r x)$, then the following scaling property holds:
\begin{equation}
  \label{eq:scfr}
  \Delta^{\alpha/2}\Phi_r(x)=r^\alpha \Delta^{\alpha/2}\Phi(rx)\,,\quad
  x\in \Rd\,.  
\end{equation}
In this respect, $\Delta^{\alpha/2}$ behaves like differentiation of order $\alpha$.

Let $\Cda=
\Gamma\left(\frac{d}{2}\right)/\left[2^\alpha\Gamma\left(\frac{d+\alpha}{2}\right)\Gamma\left(1+\frac{\alpha}{2}\right)\right]$ and define
\begin{align}\label{eq:weetb}
S(x)&=\Cda \left(1-|x|^2\right)^{\alpha/2}, \qquad \mbox{if} \quad |x|\le 1,
\end{align}
and $S(x)=0$ if $|x|>1$. We then have 
\begin{align}\label{eq:Seet}
\Delta^{\alpha/2} S(x)=-1,\qquad |x|<1. 
\end{align}
Indeed, $S(x)$ is identified in \cite{MR0137148} with the expected time of the first exit from the unit ball for the isotropic $\alpha$-stable L\'evy process starting at $x\in \Rd$, from which \eqref{eq:Seet} follows, cf. \cite[Lemma~5.3]{MR1825645} and \cite[Lemma~3.8]{MR1671973}.
An alternative approach to \eqref{eq:weetb}
and further probabilistic connections may be found in \cite[(5.4)]{MR1825645} and \cite{MR0137148}. A direct purely analytic proof of \eqref{eq:Seet} is given in \cite{MR2974318}, cf. Table~3 ibid. 

For $\lambda\in ({-d}, \alpha)$ we consider, after \cite[Section 5]{MR2075671}, the following kernel
  \begin{align}\label{ugamma}
    u_\lambda(t)&=
    \int_0^\infty r^{d+\lambda-1}(r^2-2rt+1)^{-(d+\alpha)/2}dr,\\
&= {\int_0^1 r^{-1}(r^{d+\lambda}+r^{\alpha-\lambda}) (r^2-2rt+1)^{-(d+\alpha)/2} dr},\label{eq:ug01}
\quad -1\le t < 1
  \end{align}
(we drop $d$ and $\alpha$ from the notation).
Since $r^2-2rt+1=(r-t)^2+1-t^2$, we see {from \autoref{ugamma}} that $u_\lambda(t)<\infty$ if $-1\le t<1$  and $u_\lambda(1)=\infty$. By \autoref{eq:ug01},
$[-1,1)\ni t\mapsto u_\lambda(t)$ is strictly increasing.
Furthermore, for every $t\in [-1,1)$, $\lambda\mapsto u_\lambda(t)$ is strictly increasing on $[(\alpha-d)/2,\alpha)$ and $\lim_{\lambda\to \alpha}u_\lambda(t)=\infty$ (see \cite[Lemma 5.2]{MR2075671} or below).

The following result is given in \cite[ (37)]{MR2075671}.
\begin{lemma}\label{decomp}
If $\Phi$ is real-valued, homogeneous of degree 
$\lambda\in (-d,\alpha)$ on $\Rd$, bounded on $\Sd$ and $C^2$ near a point $\eta \in \Sd$, and if $\phi$ is the restriction of $\Phi$ to $\Sd$, then
\begin{align*}
    \lap \Phi(\eta)=\lapS \phi(\eta)+R_\lambda\phi(\eta),
\end{align*}
where
the spherical fractional Laplacian is 
\begin{align}
      \lapS \phi(\eta)=\Ad \lim_{\epsilon\to 0^+}
      \int_{\Sd\setminus\left\{\ps\theta\eta>\cos\epsilon \right\}} 
      [\phi(\theta)-\phi(\eta)]u_0(\ps\theta\eta)\sigma(d\theta), \label{eq:sul}
\end{align} 
and the radial 
operator of order $\lambda$ is
\begin{align}
      R_\lambda\phi(\eta)=\Ad\int_\Sd \phi(\theta)
      [u_\lambda(\ps\theta\eta)-u_0(\ps\theta\eta)]\sigma(d\theta). \label{eq:rul}
\end{align}
\end{lemma}
\noindent
We may consider the equation $\lapS \phi(\eta)=\lap \Phi(\eta)$, $\eta\in \Sd$, as the definition of $\lapS \phi$, if $\Phi$ is the $0$-homogeneous extension of $\phi$ to $\Rdz$.

For two nonnegative functions $f$ and $g$ on a set $D$, we 
{say that $f$ is comparable to $g$ and} 
write $f\asymp g$, or $f(x)\asymp g(x)$ for $x\in D$, if constant $C$ exists such that
$$C^{-1}g(x) \le f(x) \le Cg(x), \quad x\in D.$$
Here constant means a positive number independent of $x$. If not specified otherwise constants depend only on $d$ and $\alpha$. 
If we write $C=C(a,\ldots,z)$, then we mean that $C$ depends only on $a,\ldots, z$.
The actual value of a constant may change from line to line.
For instance,  $M(\theta)\asymp \dist(\theta, \Gamma^c)^{\alpha/2}$, $\theta\in \Sd$, which is related to the fact that $\Gamma$ is smooth except at the vertex, and so
$$M(x)\asymp |x|^{\beta-\alpha/2} \dist(x, \Gamma^c)^{\alpha/2}, \quad x\in\Rdz,$$  
see \cite[Lemma 3.3]{MR2213639}.
We say that $\phi$ defined on $\Sd$ is $C^2$ on a part of the sphere if its $0$-homogeneous extension to $\Rd$ is $C^2$ in a neighborhood of this part of the sphere.

The following lemma helps estimate the homogeneity degree $\beta$ of $M$. 
\begin{lemma}\label{betaestimates}
Let $\phi\in C(\Sd)\cap C^2(B_\Theta)$ and $\phi(\theta)\asymp \dist^{\alpha/2}
(\theta,\Gamma_\Theta^c)$ for $\theta\in \Sd$.\\
If $\lambda\in(0,\alpha)$ is such that $\lapS\phi+R_\lambda\phi\ge0$ on $B_\Theta$, 
then $\lambda\ge \beta(\alpha,\Theta)$. \\
If $\lambda\in(0,\alpha)$ is such that $\lapS\phi+R_\lambda\phi\le0$ on $B_\Theta$, 
then $\lambda\le \beta(\alpha,\Theta)$. 
\end{lemma}
\begin{proof} Suppose that $0<\lambda<\beta(\alpha,\Theta)$. 
Let $\Phi(0)=0$ and
 $\Phi(x)=|x|^{\lambda}\phi\left( x/|x|\right)$ for $x\in \Rdz$. 
Note that $\lap \Phi(x)=|x|^{\lambda-\alpha}\lap \Phi\left( x/|x|\right)$ for $x\in \Gamma_\Theta$.
The function $h=\Phi-M$ is continuous on $\Rd$, $C^2$ on $\Gamma_\Theta$ and vanishes on $\Gamma_\Theta^c$.
Since $\Phi$ and $M$ are comparable on $\Sd$, $h(x)<0$ for large enough $x\in \Gamma_\Theta$ and $h(x)>0$ for small enough $x\in \Gamma_\Theta$.
Therefore $h$ has a global positive maximum at some $x_0\in \Gamma_\Theta$.
Considering the integration on $\Gamma_\Theta^c$ in \autoref{eq;deful} and \autoref{eq:scfr} we obtain
\begin{align*}
0&>\lap h(x_0)=\lap \Phi(x_0)-\lap M(x_0)=\lap \Phi(x_0)\\
&=|x_0|^{\lambda-\alpha}\left[\lapS\phi(x_0/|x_0|)+R_\lambda\phi(x_0/|x_0|)\right].
\end{align*}
This yields the first assertion, and the second one is proved similarly.
\end{proof}

\subsection{Inversion}
We shall construct functions $\phi$ satisfying the assumptions of \autoref{betaestimates}
by using an appropriate Kelvin transform. 
The inversion $T$ with respect to the unit sphere $\Sd$ is 
$$
Tx=\frac{1}{|x|^2}x,\quad x\in \Rdz.
$$
Note that $T^2$ is the identity of $\Rdz$, and $T\left(\frac12\one\right)=2\one$.
The inversion preserves angles
and the class of all straight lines and circles on $\Rdz$, because
\begin{equation}\label{conf}
|Tx-Ty|=\frac{|x-y|}{|x||y|},\quad x,y\in \Rdz.
\end{equation}
The Kelvin transform $K$ appropriate for $\lap$ is defined, for functions $v$ on $\Rdz$, as follows
$$
(Kv)(y)=|y|^{\alpha-d}v(Ty), \qquad y\in \Rdz.
$$
Thus, $K^2v=v$. We have
\begin{equation}\label{lapKt}
\lap [Kv](y)=|y|^{-\alpha-d}\lap v(Ty), {\qquad y\in \Rdz}.
\end{equation}
The formula is given in \cite[p.~112]{MR2256481} as a consequence of a transformation rule for Green potentials of $\lap$, cf. (71) and (72) ibid.
In particular, if $v$ is $\alpha$-harmonic on open set $D\subset \Rdz$, then $Kv$ is $\alpha$-harmonic on $TD$.
We define the Riesz kernel 
\begin{equation}\label{eq:h}
h(y)=|y|^{\alpha-d},
\end{equation}
and recall that $h$ is $\alpha$-harmonic on $\Rdz$. In fact, $h=K1$.

For $y=(y_1,\ldots,y_{d-1},y_d)\in\Rd$ we let $\ty=(y_1,\ldots,y_{d-1})\in \Rdd$, so that $y=(\ty,y_d)$. 
We consider the shifted cone 
\begin{gather*}
  V_\Theta=\Gamma_\Theta+\frac12\one,
\end{gather*}
and the sphere of radius $1/2$ centered at $\frac12\one$, which we denote by
$$\Sp=\left\{z\in \Rd: |z-\frac12\one|=\frac12\right\},$$
see \autoref{f1a}.
\begin{figure}[t]
\subfloat[cone]{
\label{f1a}
\begin{tikzpicture}[scale=3.1]
  \clip (-1.1,-0.2) rectangle (1.1,2.1);
  \draw (0,0) coordinate (cinv) circle (1);
  \draw (0,0.5) coordinate (c) circle (0.5);
  \draw (-0.42,0.65) node {\small $\Sp$};
  \draw (-1,1) -- (1,1);
  \fill (0,0) circle (0.01) node [below] {\tiny $0$};
  \fill (0,0.5) circle (0.01) node [left] {\tiny $\frac12\one$};
  \fill (0,1) circle (0.012) node [below] {\tiny $\one$};
  \draw[white] (c) +(45:-1.5) -- ++(45:0.5) coordinate (r1) -- ++(45:1);
  \draw[white] (c) +(135:-1.5) -- ++(135:0.5) coordinate (r2) -- ++(135:1);
  \fill[black!80,fill opacity=0.1] (r2) ++(135:1) arc (135:45:1.5) -- (c) -- cycle; 
  \draw[very thick] (r2) arc (135:45:0.5);

  \draw (c) +(60:-1.5) -- ++(60:0.2) coordinate (yinv) -- ++(60:1);;
  \fill (yinv) circle (0.01) node [above left=-2pt] {\tiny $T\!y$};

  \node (kolo) at ({-sqrt(3)},1) [circle through=(cinv)] {};
  \coordinate (y) at (intersection 1 of kolo and 0,0--yinv);
  \draw[densely dashed] (0,0) -- (y);
  \fill (y) circle (0.01) node [above left=-3pt] {\tiny $y$};

  \fill (c) +(60:0.5)  circle (0.014) node [below=2pt] {\tiny $\;\;T\!y^{\!*}$};

  \fill (-0.72,1.38) node {\small $V_\Theta$};
 \fill (-0.73,0.88) node {\small $\Gamma_\Theta$};
  \fill (0.85,0) node {\small $\Sd$};
\fill (-0.9,1.8) node [above] {{$\Rd$}};

  \draw[black] (0,0) +(45:0) -- ++(45:0.25) coordinate (rr1) -- ++(45:1);
  \draw[black] (0,0) +(135:0) -- ++(135:0.25) coordinate (rr2) -- ++(135:1);
\end{tikzpicture}
}
\subfloat[spindle]{
\label{f1b}
\begin{tikzpicture}[scale=3.1]

  \clip (-1.8,-0.17) rectangle (0.75,2.1);
  \draw[black] (0,0) +(45:0) -- ++(45:0.05) coordinate (rr1) -- ++(45:1);
  \draw[black] (0,0) +(135:0) -- ++(135:0.25) coordinate (rr2) -- ++(135:1);

  \draw (0,0) circle (1);
  \fill (0,0) circle (0.01) node [below] {\tiny $0$};
  \fill (0,1) circle (0.012) node [below] {\tiny $\one$};
  \fill (-0.6,1) node [above] {\small $F$};
  \fill (0,2) circle (0.01) node [below] {\tiny $2$}; 
  \fill (-1.1,1.8) node [above] {{$\Rd$}};
  \draw[white,densely dashed] (0,0) -- (intersection of 0,0--r1 and -1,1--1,1) coordinate (p1);
  \draw[white, densely dashed] (0,0) -- (intersection of 0,0--r2 and -1,1--1,1) coordinate (p2);
  \draw[very thick] (p1) -- (p2);
  \draw[white] (p1) +(-1.4142,0) circle (1.4142);
  \draw[white] (p2) +(1.4142,0) circle (1.4142);
  \fill[black!80,opacity=0.1] (0,0) arc (-45:45:1.4142) arc (135:215:1.4142); 
  \draw (p1)+(0,-1)--+(0,1);
\draw (p2)+(0,-1)--+(0,1);
  \draw ({-sqrt(3)},1) circle (2);
  \draw[densely dashed] (0,0) -- (y);
  \fill (y) circle (0.01) node [above left=-2pt] {\tiny $(\widetilde{y},y_d)=y$};

  \coordinate (x) at (intersection 1 of kolo and -1,1--1,1);
  \fill (x) circle (0.012) node [above right=-3pt] {\tiny $y^*=(\widetilde{y^*},1)$};
  \fill (-0.15,0.65) node {\small $L_\varepsilon$}; 
 \fill (0.35,0) node {\small $\Pi_\varepsilon$};

  \clip (-2,0.8) rectangle (1,2.1);
  \draw (-2,1) -- (1,1);
  \draw[very thick] (p1) -- (p2);
  \draw ({-sqrt(3)},1) circle (2);

  \fill ({-sqrt(3)},1) coordinate (r) circle (0.01);
  \draw ({-sqrt(3)},1) node [below] {\tiny $a$};
  \draw[dashed] 
       (r) -- (y) node  [pos=0.5,sloped,below] {\tiny $r$};
\draw (-0.95,1) arc (0:20:0.5);
  \draw (-1.05,1.03) node  [above=-1pt] {\tiny $\gamma$};
 \fill (-0.73,0.88) node {\small $\Gamma_\Theta$};
\end{tikzpicture}
}
\caption{
The shifted cone $V_\Theta$ and its inversion $L_{\varepsilon}$
}\label{fig:1}
\end{figure}
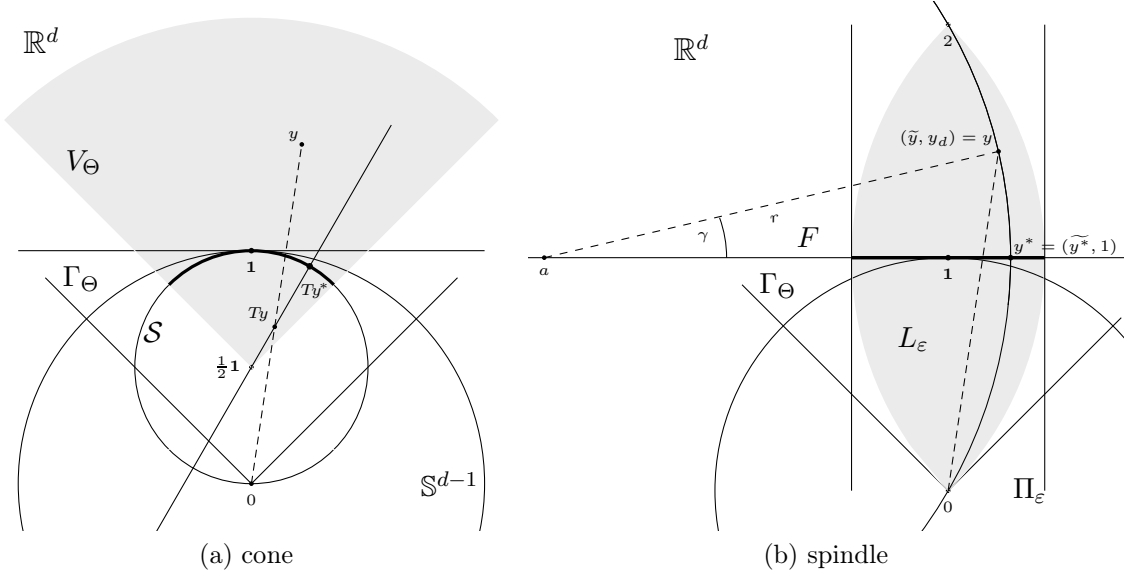
In particular, as shown on \autoref{f1b}, the inversion of $\Sp$ is flat:
$$F:=T\Sp=\{y=(\ty,y_d):y_d=1\}.$$ 
  Let $\varepsilon$ be the radius of $T(\Sp\cap V_\Theta)$, that is
\begin{equation}\label{eoT}
  \varepsilon=\tan \frac{\Theta}{2}.
\end{equation}
Since \autoref{mainthm} is asymptotic, in what follows we may and do assume that $\varepsilon<1/20$.
We consider the cylinder 
$$
\Pi_\varepsilon=\{y=(\ty,y_d):|\ty|< \varepsilon\},
$$
and the spindle-shaped image of $V_\Theta$ by $T$:
$$L_\varepsilon=TV_\Theta,$$ 
see \autoref{f1b}, which is tangent to the boundary of $\Pi_\varepsilon$, cf. \autoref{eoT}. 
For $y\in L_\varepsilon$, we denote by $y^*$ the intersection point of $F$, $\Pi_\varepsilon$ and the circle (or line) passing through $0$, $2\one$ and $y$. Thus, $y^*$ is a curvilinear projection of $y$ on $F$. Equivalently, $Ty^*$ is the intersection point of $\Sp$ and the ray from $\frac12\one$ through $Ty$. Note that $y^*_d=1$. We claim that
\begin{align}\label{qbt}
  \big||\ty|-|\tys|\big|&\le \frac{1}{2} |y-y^*|^2, \quad y\in \Rdz.
\end{align}
Indeed, if $\ty=0$, then $\tys=0$, and \autoref{qbt} is trivial. Else $0$, $2\one$ and $y$ are not collinear, and we let $a$ denote the center of the circle through these points, $r$ its radius and $\gamma$ the angle between the lines $ay$ and $ay^*$. We then observe that $r\ge 1$, $|y-y^*|=2r\sin(\gamma/2)$,
and
$\big||\ty|-|\tys|\big|=r-r\cos \gamma=2r\sin^2(\gamma/2)=
|y-y^*|^2/(2r)\le 
|y-y^*|^2/2$, as claimed, cf. \autoref{f1b}.
We let 
\begin{align*}
  s_\varepsilon(y)=
\Cdmja(\varepsilon^2-|\ty|^2)_+^{\alpha/2},\quad y\in \Rd.
\end{align*}
We have 
\begin{align}\label{flpe}
\Delta^{\alpha/2} s_\varepsilon(y)&=-1, \quad y\in \Pi_\varepsilon.
\end{align}
Indeed, $\widetilde{X_t}$ is the isotropic $\alpha$-stable L\'evy process in $\Rdd$, and
the first exit time of $\widetilde{X_t}$ from $\widetilde B$ is
$s_\varepsilon$,  the same as the expected exit time of $X_t$ from $\Pi_1$.
This yields \eqref{flpe}, cf. \cite[Lemma~5.3 and p. 319]{MR1825645} and \cite[Lemma~3.8]{MR1671973}.
For $x,y\in \Rd$ we have
\begin{equation}\label{mcse}
|s_\ve(y)-s_\ve(x)|\le 2\Cdmja\varepsilon^\alpha||\ty/\ve|-|\tx/\ve||^{\alpha/2}\le 2\Cdmja \ve^{\alpha/2}|y-x|^{\alpha/2}.
\end{equation}
We also define 
\begin{align*}
  s^*_\varepsilon(y)&=
  \begin{cases}
\Cdmja(\varepsilon^2-|\tys|^2)_+^{\alpha/2},
&\quad \text{ if } y\in L_\varepsilon,\\
  0,&\quad\text{otherwise.}
  \end{cases}
\end{align*}

\subsection{Main estimates}
In this section we present a chain of estimates.
As we shall see in \autoref{sec:pmr}, they lead to functions $\phi$ satisfying the assumptions of 
\autoref{betaestimates}, and so yield \autoref{mainthm}.

Recall that $\omega$ is defined in \autoref{omega} and the Riesz kernel $h$ is defined in \autoref{eq:h}.
\begin{lemma}\label{tpcdc}
  For $x\in \Pi_\varepsilon\cap F$ we have
  $\lap (hs_\varepsilon) (x)=-h(x)+O(\omega(\varepsilon))$.
\end{lemma}
\noindent
The proof of \autoref{tpcdc} is given in \autoref{sec:ug}.
\begin{lemma}\label{ohsss}
  For $x\in \Pi_\varepsilon\cap F$ we have
$\lap [h(s_\varepsilon^*-s_\varepsilon)](x)=O(\varepsilon^{\alpha})$.
\end{lemma}
\noindent
The proof of \autoref{ohsss} is given in \autoref{sec:ug}.
We define
\begin{equation}\label{shf}
u=K(hs^*_\varepsilon),
\end{equation}
or 
\begin{align}\label{eq:idu}
u(y)=s^*_\varepsilon(Ty).
\end{align}
Note that $u$ is supported on $V_\Theta$ and constant on rays from $\frac{1}{2}\one$, since $s_\varepsilon^*(y)=s^*_\varepsilon(y^*)$ for all $y\in \Rdz$.
If $x\in \Sp\cap V_\Theta$, then $|x-\one/2|=1/2$ and $x^*=x$. Note that $|Tx|=1/|x|$ and, by \autoref{conf}, $|Tx-\one|=|Tx-T\one|=|x-\one|/|x|$. This simplifies \autoref{shf} as follows 
\begin{align*}
  u(x)&=
  \cda(\varepsilon^2-|Tx-\one|^2)_{+}^{\alpha/2}
  =
  \cda\left(\varepsilon^2-\frac{|x-\one|^2}{|x|^2}\right)_{+}^{\alpha/2},\qquad x\in \Rdz.
  \end{align*}
For $\theta\in \Sd$ we define the profile function,
\begin{align}\label{eq:defphi}
  \phi(\theta)&:=2^\alpha u(\theta/2+\one/2)=
  2^\alpha \cda\left(\varepsilon^2-\frac{|\one-\theta|^2}{|\one+\theta|^2}\right)_{+}^{\alpha/2}.
\end{align}
\begin{proposition}\label{lapu}
For 
$\theta\in \Sp\cap V_\Theta$ we have 
\begin{align*}
\lap u(\theta)=
-|\theta|^{-2\alpha}+O(\omega(\varepsilon)).
\end{align*}
\end{proposition}
\begin{proof}
  Note that $\varepsilon^\alpha= \BigO(\omega(\varepsilon))$. By \autoref{tpcdc}, \autoref{ohsss},
\autoref{lapKt} and \autoref{qbt} we get the result.
\end{proof}

\begin{corollary}\label{sphericalest}
We have $\lapS \phi(\eta)=-1+\BigO(\omega(\Theta))$ for $\eta\in B_\Theta$.
\end{corollary}
\begin{proof}
By \autoref{lapu} for $\theta\in 
\Sp
\cap V_\Theta$, we have
  \begin{align*}
    \lap u(x)=-1+\BigO(\omega(\varepsilon)).
  \end{align*}
We consider
$x\mapsto 2^\alpha u(x/2+\one/2)$.
  The function is homogeneous of order $0$.
By \autoref{eq:scfr}
    $\lapS \phi(\theta)=
(\lap u)(\theta/2+\one/2)=-1+\BigO(\omega(\Theta))$.
\end{proof}

To verify the assumptions of \autoref{betaestimates} 
we need to
estimate $R_\lambda \phi$, cf. \autoref{eq:rul}.
\begin{lemma}\label{bound3} 
If
$0<\lambda<\alpha<\lambda+1$ and $0<\delta<1$, then for  $c=c(d)$ and $C=C(d)$,
\begin{gather*}
\frac1{\alpha-\lambda}-c\le
u_\lambda(t)-u_0(t)\le                
\frac1{\alpha-\lambda}
+\frac{C}{(\alpha-\lambda)^{1-\delta}}+\frac{C{[{1\vee}} (1-t)^{-(d+\alpha-3)/2}{]}}{(\alpha-\lambda)^{\delta}}, \quad t\in[-1,1).
\end{gather*}
\end{lemma}
\noindent
The proof of \autoref{bound3} is
given in \autoref{sec:ug}.

\begin{lemma}\label{lem:intexit} Let $ \widetilde C_{d,\alpha}=\frac12\omega_{d-1}\cda B\left(1+\frac{\alpha}{2},\frac{d-1}{2}\right)$. We have  
  \begin{align*}
    \int_{\Sd} \phi(\theta)\sigma(d\theta)=
    \widetilde C_{d,\alpha}\Theta^{\alpha+d-1}\left[1+\BigO(\Theta^2)\right],\qquad \mbox{ as } \quad \Theta\to 0^{+}.
  \end{align*}
\end{lemma}
\noindent
The proof of \autoref{lem:intexit} is given in \autoref{sec:ug}.

\begin{lemma}\label{lem:Rlgd}
Let 
$B_{d,\alpha} = \Ad \widetilde C_{d,\alpha}$, as in \eqref{eq:defs}.
For all $\eta\in B_\Theta$ we have
    \begin{align}\label{radiallower}
      R_\lambda\phi(\eta)\ge  B_{d,\alpha}
          \Theta^{d-1+\alpha}\left[1+\BigO(\Theta^2)\right]\frac{1-c
(\alpha-\lambda)}{\alpha-\lambda},
    \end{align}
and, under the condition $0<\delta<1$, 
    \begin{align}\label{radialupper}
      R_\lambda\phi(\eta)\le B_{d,\alpha} \Theta^{d-1+\alpha}\left[1+\BigO(\Theta^2)\right]
        \frac{1+C(\alpha-\lambda)^{\delta}}{\alpha-\lambda}
        +C \frac{\Theta\omega(\Theta)}{(\alpha-\lambda)^{\delta}}.
    \end{align}
\end{lemma}
\noindent
The proof of \autoref{lem:Rlgd} is given in \autoref{sec:ug}.

\subsection{Proof of \autoref{mainthm}}\label{sec:pmr}
Let 
\begin{align*}
  \lambda=\alpha-{B_{d,\alpha}}\Theta^{d-1+\alpha}[1+\kappa\omega(\Theta)],
\end{align*}
where $\kappa>0$ shall be defined later.
We let $\delta=(d+\alpha-1)^{-1}$ in \autoref{radialupper} and obtain
\begin{align*}
  R_\lambda\phi(\eta)&\le   \left(1+c_1\Theta^2\right)
    \frac{1+c_2\Theta}{1+\kappa\omega(\Theta)}
    +c_3 \omega(\Theta),\qquad \eta\in B_\Theta.
\end{align*} 
If $0\le a\le b<1$, then $(1+a)/(1+b)\le 1- (b-a)/2$. 
If $\kappa\ge c_2$ and $\kappa\omega(\Theta)< 1$, then
\begin{align*}
R_\lambda\phi(\eta)&\le 1+(-\kappa/2+c_2/2+c_1+c_3)\omega(\Theta),\qquad \eta\in B_\Theta.
  \end{align*}
By \autoref{sphericalest},
\begin{align*}
  \lapS \phi(\eta) &\le
  -1+c_4 \omega(\Theta),\qquad \eta\in B_\Theta.
\end{align*}
Accordingly, we stipulate $\kappa\ge 2c_1+c_2+2c_3+2c_4$. For $\omega(\Theta)< 1/\kappa$ we then have
\begin{align*}
  \lapS \phi(\eta)+R_\lambda\phi(\eta)&\le 0,\qquad \eta\in B_\Theta,
\end{align*}
and \autoref{betaestimates} yields
\begin{align*}
  \beta(\alpha,\Theta) \ge \alpha - {B_{d,\alpha}}\Theta^{d-1+\alpha}[1+\kappa\omega(\Theta)].
\end{align*}
To obtain an opposite bound, we put
\begin{align*}
  \lambda=\alpha-{B_{d,\alpha}}\frac{\Theta^{d-1+\alpha}}{1+\iota\omega(\Theta)},
\end{align*}
and we shall define $\iota>0$ momentarily. By \autoref{radiallower},
\begin{align*}
R_\lambda\phi(\eta)
  &\ge
      \frac{1-c(\alpha-\lambda)}{\alpha-\lambda}  B_{d,\alpha} \Theta^{d-1+\alpha}\left[1+\BigO(\Theta^2)\right]
  \\&\ge
   1-c' \Theta^{d-1+\alpha}+\iota\omega(\Theta)-c''\Theta^2.
\end{align*} 
Recall that $d\ge 2$ and $\Theta\le \omega(\Theta)$. Taking $\iota\ge c'+c''$, by \autoref{sphericalest} we get
\begin{align*}
  \lapS \phi(\eta)+R_\lambda\phi(\eta)\ge0,
\end{align*}
and \autoref{betaestimates} yields
\begin{align*}
  \beta(\alpha,\Theta) \le 
      \alpha - B_{d,\alpha}\Theta^{d-1+\alpha}
      \left[1-\iota\omega(\Theta)/2\right],
\end{align*}
provided $\omega(\Theta)<1/(2\iota)$. This ends the proof of \autoref{mainthm}.
\section{
Technical details
}\label{sec:ug}
We now give proofs of the more technical lemmas from \autoref{sec:prel}, and further results.

\subsection{Proof of \autoref{tpcdc}}
For $x\in \Pi_\varepsilon\cap F$ we have
\begin{align}
\lap(hs_\varepsilon)(x)&=-h(x)+
\left[\lap(h s_\varepsilon)(x)-h(x)\lap s_\varepsilon(x)-s_\varepsilon(x)\lap h(x)\right]\nonumber
\\&=-h(x)+\limeps\Ad\int_{|x-y|\ge \epsilon} 
\frac{[s_\varepsilon(y)-s_\varepsilon(x)][h(y)-h(x)]}{|x-y|^{d+\alpha}}dy.\label{mult}
\end{align}
To analyze the integral in \autoref{mult}, we define the following sets 
\[G=\{y\in\Rd:|\ty|<1/2,|y_d|\le1/2\},\] 
\[H=\{y\in\Rd :|\ty|<1/2, |y_d-1|\le 1/2 \}.\] 
Recall that $s_\varepsilon(y)\le c_{d-1,\alpha}\ve^\alpha$, $y\in \Rd$. 
Observe that $h(x)\le 1$.
For $y\in (G\cup H)^c$ we have $|y|>1/2$ and $|x-y|>c|y|$, hence 
\begin{align*}
  \int_{(G\cup H)^c} \frac{|s_\varepsilon(x)-s_\varepsilon(y)||h(x)-h(y)|}{|x-y|^{d+\alpha}}dy\le
  c \varepsilon^{\alpha}\int_{(G\cup H)^c}  (|y|^{\alpha-d}+1)|y|^{-d-\alpha} dy = c\varepsilon^{\alpha}.
\end{align*}
On $G$ we have $|x-y|\ge 1/2$, and
\begin{align*}
  \int_{G} |s_\varepsilon(x)-s_\varepsilon(y)||h(x)-h(y)||x-y|^{-d-\alpha}dy\le 
  \cda \ve^{\alpha} 2^{d+\alpha} \int_G h(y) dy\le c\varepsilon^{\alpha}.
\end{align*}

Note that $H\subset B(x,1)$.
Let $\delta = \ve-|\tx|$, the distance from $x$ to $\Pi_{\ve}^c$. Then 
$s_\ve(x) \le c\delta^{\alpha/2}\ve^{\alpha/2}$, $x\in\Pi_\ve\cap F$.
We split the integral on $H$ into $H_1= H\setminus\Pi_\ve$ and $ H_2=H\cap\Pi_\ve$. Since $h$ is Lipschitz and $s_\ve(y) = 0$ on $H_1$, we get
\begin{align*}
I_1:&=\int_{H_1} \frac{|s_\varepsilon(x)-s_\varepsilon(y)||h(x)-h(y)|}{|x-y|^{d+\alpha}}dy
  \le 
c\delta^{\alpha/2}\ve^{\alpha/2}\int_{H_1} \frac{1}{|x-y|^{d-1+\alpha}}dy
\\&\le c\delta^{\alpha/2}\ve^{\alpha/2}\int_{B(x,1)\setminus B(x,\delta)} \frac{1}{|x-y|^{d-1+\alpha}}dy
\end{align*}
If $\alpha<1$, then the last integral is bounded by 
\[ \int_{B(x,1)} \frac{1}{|x-y|^{d-1+\alpha}}dy \le c\]
and, $\delta\le \ve$, implies $I_1\le c\ve^{\alpha}$.
If $\alpha>1$, then
\begin{align*}
I_1 &\le c\delta^{\alpha/2}\epsilon^{\alpha/2}
\int_{B(x,\delta)^c}|x-y|^{-d-\alpha+1}dy 
\\& \le 
c\delta^{\alpha/2}\epsilon^{\alpha/2} \delta^{1-\alpha}
\\& = c\delta^{1-\alpha/2}\epsilon^{\alpha/2}.
\\& \le c\ve^{1-\alpha/2}\epsilon^{\alpha/2} = c\ve.
\end{align*}
Finally, for $\alpha = 1$ we have 
\begin{align*}
I_1 &\le 
c\delta^{\alpha/2}\ve^{\alpha/2}
\int_{B(x,1)\setminus B(x,\delta)} |x-y|^{-d-\alpha+1}dy
\\& \le c\delta^{\alpha/2}\ve^{\alpha/2}|\ln\delta|.
\end{align*}
But $\sqrt\delta|\ln\delta|$ is increasing on $0<\delta<e^{-2}$, while we have $\delta\le \varepsilon<1/20$. Hence we can replace $\delta$ with $\varepsilon$ in the last line. Summarizing, for any $\alpha\in(0,2)$ we have $I_1\le c\omega(\varepsilon)$.

Recall that by \eqref{mcse}, for any $x$ an $y$, we have $|s_\ve(x) - s_\ve(y)|\le c\ve^{\alpha/2}|x-y|^{\alpha/2}$. Hence, 
\begin{align*}
  \int_{H_2} \frac{|s_\varepsilon(x)-s_\varepsilon(y)||h(x)-h(y)|}{|x-y|^{d+\alpha}}dy  &\le \int_{B(x,1)} \frac{|s_\varepsilon(x)-s_\varepsilon(y)||h(x)-h(y)|}{|x-y|^{d+\alpha}}dy
 \\ & \le 
 c\ve^{\alpha/2}\int_{B(x,1)} |x-y|^{-d+(1-\alpha/2)}dy 
\\&\le c\ve^{\alpha/2}\delta^{1-\alpha/2}\le c \varepsilon.
\end{align*}
But $\varepsilon^\alpha\le \omega(\varepsilon)$ and $\varepsilon\le \omega(\varepsilon)$, ending the proof of \autoref{tpcdc}.

\subsection{Proof of \autoref{ohsss}}

For $x\in \Pi_\varepsilon\cap F$ we have $x^*=x$ and $s_\varepsilon^*(x)-s_\varepsilon(x)=0$. Furthermore, for $y\in\Pi_\ve$ we have $s_\varepsilon^*(y) \le s_\varepsilon(y)$.
Therefore, 
\begin{align}
\lap [h(s_\varepsilon^*-s_\varepsilon)](x)&=
\limeps\Ad\int_{\Pi_\varepsilon\cap\{|x-y|\ge\epsilon\}} h(y)[s_\varepsilon(y)-s^*_\varepsilon(y)]|x-y|^{-d-\alpha}dy\nonumber\\
&=
\Ad\int_{\Pi_\varepsilon} h(y)[s_\varepsilon(y)-s^*_\varepsilon(y)]|x-y|^{-d-\alpha}dy
\label{bezPV}.
\end{align}
Before we estimate this last integral we need to introduce a new geometric context.

We simplify the notation by centering at $0$, so that $F$ becomes $\{z\in\Rd : z_d=0\}$ and can be identified with $\Rdd$. 
Namely, if $z\in B(0,1)$, then we consider
the circle (or line) 
passing through $\one$, $-\one$, and $z$, which intersects the hyperplane $\{z_d=0\}$ at $z^*=(\tzs,0)$, a unique point in $B(0,1)$.
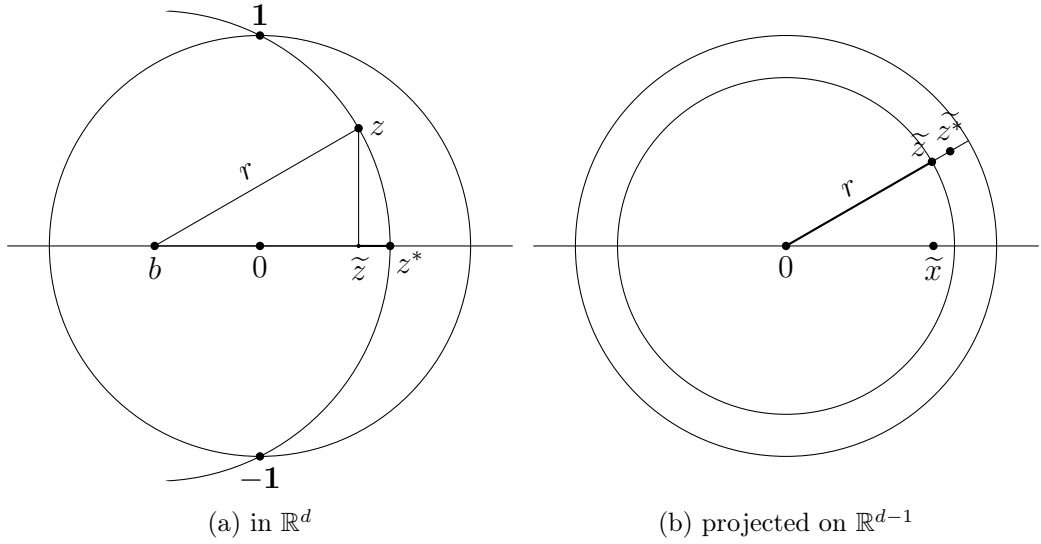
\begin{figure}[t]
\centering
\subfloat[in $\Rd$]{
\begin{tikzpicture}[scale=2.8]
  \clip (0,0) circle (1.2);
  \draw (-2,0) -- (2,0);
  \draw (0,0) circle (1);
  \fill (0,1) circle (0.02) node [above] {$\one$};
  \fill (0,-1) circle (0.02) node [below] {$-\one$};
  \fill (0,0) circle (0.02) node [below] {$0$};
  \fill (-1/2,0) circle (0.02) node [below] {$b$};
  \fill ({-1/2+sqrt(5)/2},0) coordinate (xs) circle (0.02) node [below right=-2pt] {$z^*$};
  \draw (-1/2,0) coordinate (b) circle ({sqrt(5)/2});
  \draw (b) -- ++({sqrt(15)/4},{sqrt(5)/4}) coordinate (x) node [pos=0.5,sloped,above] {$r$};
  \fill (x) circle (0.02) node [right] {$z$};
  \draw (b) -- ++({sqrt(15)/4},0) coordinate (tx);
  \fill (tx) circle (0.01) node [below] {$\tz$};
  \draw (x) -- (tx); 
  \draw[thick] (tx) -- (xs);
\end{tikzpicture}
\label{dzzsa}

}
\subfloat[projected on $\Rdd$]{
\begin{tikzpicture}[scale=2.8]
  \clip (0,0) circle (1.2);
 \draw (0,0) circle (1);
  \fill (0,0) circle (0.02) node [below] {$0$};
  \fill (0.7,0) circle (0.02) node [below] {$\widetilde x$};
  \fill (30:0.8) circle (0.02) node [above left=-2pt] {$\widetilde z$};
  \draw (0,0) circle (0.8);
  \fill (30:0.9) circle (0.02) node [above] {$\widetilde{z^*}$};
  \draw (-2,0) -- (2,0);
  \draw (0,0) -- (30:1);
  \draw[thick] (0,0) -- (30:0.8) node[pos=0.5,above,sloped] {$r$};
\end{tikzpicture}
\label{dzzsb}
}
\caption{Coordinates $z_d$, $z^*$.}
\end{figure}
The situation is shown on \autoref{dzzsa}.  
We denote by $r$ and
$b=(\widetilde b,0)$ the radius and the center of the circle. Namely,
\begin{align}\label{dbr}
  b=-z^*\frac{1-|z^*|^2}{2|z^*|^2},\qquad
  r=\frac{1+|z^*|^2}{2|z^*|},
\end{align}
because direct verification gives that
\begin{align*}
  r^2 = |b-\one|^2 = |b-z^*|^2 = |b-z|^2 = |b+\one|^2.
\end{align*}
Of course, $r>1$. In particular,
$r^2=|b-\widetilde z|^2+z_d^2=(|b|+|\widetilde z|)^2+z_d^2$,
and so
$|\widetilde z|=\sqrt{r^2-z_d^2}-|b|$.
We see that
\begin{align*}
  |z^*|-|\widetilde z|=|z^*|+|b|-\sqrt{r^2-z_d^2}=r-\sqrt{r^2-z_d^2}=
  \frac{z_d^2}{r+\sqrt{r^2-z_d^2}}.
\end{align*}
By \autoref{dbr}, $ |z^*|/\sqrt{2} \le 1/r \le \sqrt{2}|z^*|$. Therefore,
\begin{equation}\label{dizzs} 
z_d^2|z^*|/(2\sqrt{2})\le |z^*|-|\widetilde z| \le \sqrt{2}\,z_d^2|z^*|,\qquad z\in B(0,1).
\end{equation}
We note that $|z^*|^2-|\tz|^2=(|z^*|+|\tz|)(|z^*|-|\tz|)$, therefore
\begin{equation}\label{dizzsd} 
z_d^2|z^*|^2/(2\sqrt{2})\le |z^*|^2-|\widetilde z|^2 \le 2\sqrt{2}\,z_d^2|z^*|^2,\qquad z\in B(0,1).
\end{equation}
\begin{lemma}\label{hollip} For $z\in L_\ve$ we have
\begin{align*}
  (\varepsilon^2-|\widetilde z|^2)^{\alpha/2}-(\varepsilon^2-|z^*|^2)^{\alpha/2}&=
  \varepsilon^{\alpha}\left(\left[1-(|\widetilde z|/\varepsilon)^2\right]^{\alpha/2}-\left[1-(|z^*|/\varepsilon)^2\right]^{\alpha/2}\right)
  \\&\le
  (|z^*|^2-|\widetilde z|^2) \left(\left[\varepsilon^2-|z^*|^2\right]^{\alpha/2-1}\wedge \left[|z^*|^2-|\widetilde z|^2\right]^{\alpha/2-1}\right).
\end{align*}
\end{lemma}
\begin{proof}
Recall that $0<\alpha<2$. We now analyze H\"older continuity of $s_\ve$. Note that
\begin{align*}
  t^{\alpha/2}=\frac{\alpha}{2}\int_0^t r^{\alpha/2-1}dr, \quad t\ge 0.
\end{align*}
Therefore, if $0<t<s$, then
\begin{align*}
  t^{\alpha/2}-s^{\alpha/2}=s^{\alpha/2}\frac{\alpha}{2}\int_1^{t/s} r^{\alpha/2-1}dr\le s^{\alpha/2}\frac{\alpha}{2}\int_1^{t/s} 1dr=
  \frac{\alpha}{2}(t-s)s^{\alpha/2-1},\\
  t^{\alpha/2}-s^{\alpha/2}=\frac{\alpha}{2}\int_s^t r^{\alpha/2-1}dr\le 
  \frac{\alpha}{2}\int_0^{t-s} r^{\alpha/2-1}dr = 
  	(t-s)^{\alpha/2},
\end{align*}
and so
\begin{align*}
  t^{\alpha/2}-s^{\alpha/2}\le
  (t-s)\left(s^{\alpha/2-1}\wedge (t-s)^{\alpha/2-1}  \right).
\end{align*} 
\end{proof}
We remark that the above is sharp, meaning that a proportional lower bound holds, too.
Indeed, let $f(x)=x^{\alpha/2}-1$ if $x\ge 1$. If $x\in[1,2]$, then $f'(x)=(\alpha/2) x^{\alpha/2-1}\ge \alpha/4$, and so $f(x)\geq \alpha(x-1)/4$.
Note that $2^{-\alpha/2} - 1 \ge \alpha/4$.
If $x\ge 2$,  then $f(x)\ge x^{\alpha/2}-(x/2)^{\alpha/2}=
x^{\alpha/2}(1-2^{-\alpha/2})\ge (x-1)^{\alpha/2}2^{-\alpha/2}(2^{\alpha/2}-1)\ge (x-1)^{\alpha/2}\alpha/8$. Therefore,
\begin{align*}
  t^{\alpha/2}-s^{\alpha/2}
  &= s^{\alpha/2}\left[ \left(\frac{t}{s}\right)^{\alpha/2}-1\right] 
 \ge \frac{\alpha}{8}\,
  s^{\alpha/2}\left[\left(\frac{t}{s}-1\right)\wedge 
  \left(\frac{t}{s}-1\right)^{\alpha/2}\right] \\
  &=\frac{\alpha}{8}\,(t-s)\left[s^{\alpha/2-1}\wedge (t-s)^{\alpha/2-1}  \right].
\end{align*}

When $L_\ve$ is centered at 0, the integral \autoref{bezPV} becomes
\begin{align}\label{bezPV1}
\int_{\Pi_\ve} h(z+\one)[s_\varepsilon(z)-s^*_\varepsilon(z)]|x-z|^{-d-\alpha}dz. 
\end{align}
Recall that for $z\in\Rd$ we consider the decomposition $z=(\tz,z_d)$. We split this integral by considering subsets of $\Pi_\ve$. Let $U= \Pi_\ve\cap(\{z_d\le -3/2\} \cup  \{z_d \ge 1/2 \})$, $V = \Pi_\ve\cap\{0\le z_d\le 1/2\}$, $V_s = \Pi_\ve\cap\{-1/2 \le z_d\le 0\}$ and $W = \Pi_\ve\cap\{-3/2 \le z_d\le -1/2\}$.
Observe that for $z\in U$ we have $|x-z|\asymp z_d$, 
the function $h(z+\one)$ is bounded and $|s_\varepsilon^*(z)-s_\varepsilon(z)|\le C\ve^\alpha$ so that the integral over $U$ is clearly $O(\ve^\alpha)$. On $W$ the function $h$ is integrable and the rest of the integrand is bounded by $C\ve^\alpha$. This leaves us with the hard part, i.e. the integral over $V\cup V_1$.
Observe that on the latter set the function $h(z+\one) = |z+\one|^{\alpha-d}$ is bounded.
Hence, it is enough to estimate 
\[  \int_V \left[s_\varepsilon(z)-s^*_\varepsilon(z)\right]|x-z|^{-d-\alpha}dz \]
(by symmetry, the integral over $V_1$ enjoys the same upper bound). For some $z\in V$,  the point $z^*$ is outside of $\Pi_\varepsilon$. Still,
\begin{align*}
|z^*|=|z^*|-|\widetilde z|+|\widetilde z|\le \sqrt{2}z_d^2|z^*|+\varepsilon\le |z^*|/2+\varepsilon.
\end{align*}
Hence $|z^*|\le 2\varepsilon$.
Therefore, by \autoref{dizzsd}
\begin{align}\label{zandh}
  |z^*|^2-|\widetilde z|^2\le \sqrt{2}z_d^2|z^*|^2\le 4z_d^2\varepsilon^2.
\end{align}

\begin{lemma}\label{szd}
  For any $z\in V$, we have
  \begin{align*}
  s_\varepsilon(z)-s_\varepsilon^*(z)\le c\varepsilon^\alpha z_d^\alpha.
  \end{align*}
\end{lemma}

\begin{proof}
On $V\setminus L_\ve$, we have $|z^*|>\varepsilon$. Hence by \autoref{zandh}
\begin{align*}
  s_\varepsilon(z)-s^*_\varepsilon(z)=s_\varepsilon(z) 
=C(\varepsilon^2-|\tz|^2)^{\alpha/2}
  \le C(|z^*|^2-|\tz|^2)^{\alpha/2}
  \le C\varepsilon^\alpha z_d^{\alpha}
\end{align*}
On $V\cap L_\ve$, \autoref{hollip} and \autoref{zandh} imply
\begin{align*}
  s_\varepsilon(z)-s^*_\varepsilon(z) 
  \le C(|z^*|^2-|\tz|^2)^{\alpha/2}
  \le C\varepsilon^\alpha z_d^{\alpha}.
\end{align*}
This ends the proof.
\end{proof}

The following lemma is standard, for a detailed proof see e.g. Lemma 5.1 in \cite{MR2114264}.
\begin{lemma}
For $x\in \Rd$ we let $\delta_{\Sd}(x)=\inf\{|y-x|: y\in \Sd\}$.
If $0<\rho<1$, then there is a constant $C_\rho$ (depending on $\rho$) such that
  \begin{align}\label{std}
    \int_{\Sd} |x-y|^{-d+\rho} d\sigma(y)\le C_{\rho} \delta_{\Sd}^{\rho-1}(x),\quad x\in \Rd.
  \end{align}
\end{lemma}

In $\Rdd$, consider coordinates $\tz=(r,\theta)$, where $\theta\in \Sdd$ and $ r=|\tz|$. 
Let $S(0,R)=\{\tz\in\Rdd : |\widetilde z|=R\}=R\Sdd$ (cf. \autoref{dzzsb}). Let $\widetilde\sigma$ denote the $(d-2)$-dimensional Hausdorff measure.
\begin{corollary}\label{offcenter}
Let $\tx\in \Pi_\ve\cap F$. 
For $\rho<1$ we have 
  \begin{align*}
  \int_{S(0,r)} |\widetilde x-\widetilde y|^{-d+1+\rho} d\widetilde\sigma(\widetilde y)
  \le 
  C_\rho |r-|\widetilde x||^{-1+\rho}, \qquad r>0.
  \end{align*}
\end{corollary}
\begin{proof}
This follows from a simple change of variable in \eqref{std}.
\end{proof}
Now, set $\delta=(2-\alpha)/8$ and
 define $D=\{z\in V : 8\varepsilon z_d^{2(1-\delta)} \le \varepsilon-|\widetilde z|\}$. Assume $z\in D$. Since $|z_d|\le 1/2$, the second bound from \eqref{dizzs} yields $|z^*|\le 2|\tz| \le 2\ve$. 
This can be improved to $|z^*|\le \ve$, because if 
$|z^*|> \varepsilon$, then by \autoref{dizzs}
\begin{align*}
8\varepsilon z_d^{2-2\delta}\le \varepsilon-|\widetilde z|\le |z^*|-|\widetilde z|\le\sqrt{2}z_d^2|z^*|\le 2\sqrt{2}z_d^2\varepsilon,
\end{align*}
giving a contradiction with $z_d\le 1/2$. 
Hence, in particular, $D\subset L_\ve$.
Furthermore, by \eqref{dizzs} and the definition of $D$,
\begin{align*}
  \varepsilon-|z^*|&=(\varepsilon-|\widetilde z|)-(|z^*|-|\widetilde z|)\ge(\varepsilon-|\widetilde z|)/2+(\varepsilon-|\widetilde z|)/2-\sqrt{2}\varepsilon z_d^2\ge
  \\&\ge(\varepsilon-|\widetilde z|)/2.
\end{align*}
Using the first bound from \autoref{hollip}, \autoref{dizzsd} and the above inequality,
\begin{align*}
  \int_{D} [s_\varepsilon(z)&-s^*_\varepsilon(z)]|x-z|^{-d-\alpha}dz
  \\&\le
  C\int_{D} (|z^*|^2-|\widetilde z|^2) (\varepsilon^2-|z^*|^2)^{\alpha/2-1}|x-z|^{-d-\alpha}dz
  \\&\le
  C\varepsilon^{1+\alpha/2}\int_{ D} z_d^2 (\varepsilon-|\widetilde z|)^{\alpha/2-1}|x-z|^{-d-\alpha}dz
  \\&\le
  C\varepsilon^{\alpha}\int_{V } z_d^{\alpha+(2-\alpha)\delta}|x-z|^{-d-\alpha}dz
  \\&\le
  C\varepsilon^{\alpha}\int_{V } |x-z|^{-d+(2-\alpha)\delta}dz<C\varepsilon^\alpha.
\end{align*}
Hence we need to consider integral over $V\setminus D$.
For a fixed $x\in F\cap L_\ve$ let $A=\{z\in V: ||\tx|-|\tz||^{1-\delta}\le z_d\}$.

Using \autoref{szd} and the inequality $|x-z|\ge z_d$  we get
\begin{align*}
  \int_{A} \left[s_\varepsilon(z)-s^*_\varepsilon(z)\right] |x-z|^{-d-\alpha}dz&\le
    C\varepsilon^{\alpha}\int_{A} z_d^{\alpha}|\widetilde x-\widetilde z|^{-d+3/2}|x-z|^{-3/2-\alpha}dz
  \\&\le
  C\varepsilon^{\alpha}\int_{A} z_d^{\alpha}z_d^{-3/2-\alpha}|\widetilde x-\widetilde z|^{-d+3/2}dz
\end{align*}
In cylindrical coordinates, this can be rewritten as
\[
  C\varepsilon^{\alpha}\int_0^\ve 
  \left(\int_{S(0,r)}|\widetilde x-\widetilde z|^{-(d-1)+1/2}d\widetilde\sigma(\widetilde z)
  \int_{z_d\ge||\tx|-r|^{1-\delta}}z_d^{-3/2}dz_d\right) dr.
\]
By \autoref{offcenter}, this is bounded by
\[
  C\varepsilon^{\alpha}\int_0^\ve ||\tx|-r|^{-1/2}||\tx|-r|^{-(1-\delta)/2} dr
  \le
  C\varepsilon^{\alpha}\int_0^1 ||\tx|-r|^{-1+\delta/2}dr<C\varepsilon^\alpha. 
\]
Now let $B=\{z\in V: z_d\le |x-z|^{1+\delta}\}$. \autoref{szd} gives 
\begin{align*}
  \int_{ B} &[s_\varepsilon(z)-s^*_\varepsilon(z)]|x-z|^{-d-\alpha}dz
  \\&\le
  C\varepsilon^{\alpha}\int_{B} z_d^\alpha |x-z|^{-d-\alpha}dz
  \\&\le
  C\varepsilon^{\alpha}\int_{B(x,1)} |x-z|^{\alpha(1+\delta)} |x-z|^{-d-\alpha}dz\le C\varepsilon^\alpha.
\end{align*}
	
We still need an estimate on $E=V\cap (A\cup B\cup D)^c$, that is
\begin{align}\label{hardpart}
  E = \{z\in V : \varepsilon-|\tz| < 8\varepsilon z_d^{2-2\delta}, \, |x-z|^{1+\delta} <  z_d < ||\tx|-|\tz||^{1-\delta}\}.
\end{align}
First, we establish some geometric properties for $z\in E$.
Suppose $|\widetilde z|<|\widetilde x|$. Then 
\begin{align*}
  \varepsilon-|\widetilde z|\le 8\varepsilon z_d^{2-2\delta}\le 8\varepsilon(|\widetilde x|-|\widetilde z|)^{2(1-\delta)^2}\le 8\varepsilon(\varepsilon-|\widetilde z|)^{2(1-\delta)^2}.
\end{align*}
Since $\varepsilon<1/8$, $\varepsilon-|\widetilde z|\le \varepsilon<1$ and
$\delta< 1/4< 1-1/\sqrt{2}$, we get a contradiction. Therefore 
\begin{align*}
  |\tx|\le|\tz|\le \varepsilon \text{ on $E$}.
\end{align*}
Let $a=\varepsilon-|\widetilde x|$ and $e=|\widetilde z|-|\widetilde x|\ge0$. Note that $e\le |x-z|\le z_d^{1/(1+\delta)} \le z_d^{1-\delta}$. By assuption, we have $\delta<1/3$. 
Since $z_d^2+|\tx-\tz|^2 = |x-z|^2$ and $e\le |\tx-\tz|$, we get $e^2+z_d^2\le |x-z|^2$. Hence
\begin{align*}
  |x-z|^2-a^2&\ge e^2+z_d^2-(\varepsilon-|\widetilde z|+e)^2
  \\ &=
  z_d^2-(\varepsilon-|\widetilde z|)^2-2(\varepsilon-|\widetilde z|)e
  \\&\ge z_d^2-64\varepsilon^2 z_d^{4-4\delta}-16\varepsilon z_d^{2-2\delta}h^{1-\delta}
  \\&= z_d^2\left(1-64\varepsilon^2 z_d^{2-4\delta}-16\varepsilon z_d^{1-3\delta}\right)
  \\&\ge z_d^2\left(1-64\varepsilon^2-16\varepsilon\right).
\end{align*}
The expression in the parenthesis is positive for $\varepsilon<1/20$, giving $|x-z|\ge a$. 
Note also that
\begin{align*} 
  z_d&\ge |x-z|^{1+\delta}\ge a^{1+\delta}\\
  z_d&\le (|\widetilde z|-|\widetilde x|)^{1-\delta}=[(\varepsilon-|\widetilde x|)-(\varepsilon-|\widetilde z|)]^{1-\delta}\le a^{1-\delta}.
\end{align*}
Since $E$ is not a subset of $L_\varepsilon$, \autoref{hollip} is not applicable. We have the following estimate instead
\begin{align*}
  s_\varepsilon(z)-s_\varepsilon^*(z)
  \le s_\varepsilon(z)
  \le C \varepsilon^{\alpha/2}(\varepsilon-|\widetilde z|)(\varepsilon-|\widetilde z|)^{\alpha/2-1}
  \le C\varepsilon^{\alpha/2}\varepsilon z_d^{2-2\delta}(\varepsilon-|\widetilde z|)^{\alpha/2-1}.
\end{align*}
Note also that 
$|x-z|^{-d-\alpha} \le |x-z|^{-d} z_d^{-\alpha}
\le |x-z|^{-d+1+\alpha/2} z_d^{-1-\alpha/2-\alpha}$. Since $-d+1+\alpha/2<0$, we get
\begin{align*}
|x-z|^{-d-\alpha} \le |\tx-\tz|^{-d+1+\alpha/2} z_d^{-1-\alpha/2-\alpha}.
\end{align*}
By asumption, $\delta+\alpha<2$. Using cylindrical coordinates, we obtain, 
\begin{align*}
  \int_{E}& \left[s_\varepsilon(z)-s_\varepsilon^*(z)\right]|x-z|^{-d-\alpha}dz
  \\&\le C\varepsilon^{\alpha/2+1}
  \int_{a^{1+\delta}}^{a^{1-\delta}} z_d^{1-\alpha/2-\alpha-2\delta}\,dz_d \int_{\varepsilon\ge |\widetilde z|\ge |\widetilde x|} |\widetilde x-\widetilde z|^{-d+1+\alpha/2}(\varepsilon-|\widetilde z|)^{\alpha/2-1}\,d\tz 
  \\&\le
  C\varepsilon^{\alpha/2+1}a^{-(1+\delta)(\alpha+2\delta)/2}
  \int_0^{a^{1-\delta}} z_d^{1-\alpha/2-(\alpha+2\delta)/2} \,dz_d
  \times\\&\qquad\qquad\qquad\qquad\times
   \int_{|\tx|}^\varepsilon (\varepsilon-r)^{\alpha/2-1}
   \int_{S(0,r)} |\widetilde x-\widetilde z|^{-(d-1)+\alpha/2}\,d\sigma(\widetilde z) dr
  \\&\le
  C\varepsilon^{\alpha/2+1}a^{-(1+\delta)(\alpha+2\delta)/2}a^{(1-\delta)(2-\alpha/2-(\alpha+2\delta)/2)} 
  \int_{|\tx|}^\varepsilon (\varepsilon-r)^{\alpha/2-1}(r-|\tx|)^{\alpha/2-1} dr
  \\&=
  C\varepsilon^{\alpha/2+1}a^{-\alpha-2\delta+(1-\delta)(2-\alpha/2)} 
  (\varepsilon-|\tx|)^{\alpha-1}
  \int_0^1 r^{\alpha/2-1}(1-r)^{\alpha/2-1} dr
  \\&=
  C\varepsilon^{\alpha/2+1}
  a^{-\alpha-2\delta+(1-\delta)(2-\alpha/2)+ \alpha-1} 
  = C\varepsilon^{\alpha/2+1}a^{1-\alpha/2-\delta(4-\alpha/2)}.
\end{align*}

Since $\delta=(2-\alpha)/8$, we get $a$ in a positive power. 
Consequently, the integral decays slightly faster than $\varepsilon^{\alpha}$. 
This completes the proof of \autoref{ohsss}.

\subsection{Proof of \autoref{bound3}}
Let $t\in [-1,1)$ and $r\in [0,1]$.
By \autoref{eq:ug01},
\begin{equation}\label{eq:ru}
  u_\lambda(t)-u_0(t) 
  = 
  \int_0^1 \left( 1-r^{\lambda} \right)\left( r^{\alpha-\lambda-1}-r^{d-1} \right)
  \left(r^2-2rt +1\right)^{-(d+\alpha)/2}dr.
\end{equation}
Let $g(r)=r^2-2rt+1=(r-t)^2+1-t^2$ and  $f(r)=g(r)^{-(d+\alpha)/2}$. 
To estimate $f$ we observe that the extrema of $g(r)$ on $[0,1]$ are either $g(0)=1$, 
or $g(1)=2(1-t)$ or $g(t)=1-t^2$.
We first prove the lower bound in the statement of the lemma.
Note that {$f(0)=1$}.
We have $f\rq{}(r)=-(d+\alpha)(r-t)g(r)^{-(d+\alpha)/2-1}$, hence $f\rq{}(r)\ge -2(d+\alpha)$ if $-1\le t <0$, and 
$f\rq{}(r)\ge-(d+\alpha) (1-t^2)^{-(d+\alpha)/2-1} \ge-(d+\alpha) (4/3)^{(d+\alpha)/2+1}$ if $0\le t<1/2$. 
If $1/2\le t<1$, then $g(r)=r(r-2t)+1\leq 1$ and $f(r)\geq 1$.
Summarizing, in each case we have $f(r)\geq 1-cr$. Therefore,
\begin{align}
 &  u_\lambda(t)-u_0(t) \label{eq:lbe}
 \ge 
  \int_0^1 \left( 1-r^{\lambda} \right)\left( r^{\alpha-\lambda-1}-r^{d-1} \right)(1-cr)dr \\ 
  &= \frac{1}{\alpha-\lambda} - \frac1d -\frac1\alpha+\frac1{d+\lambda}
  -c\left(\frac{1}{\alpha-\lambda+1} - \frac1{d+1} -\frac1{\alpha+1}+\frac1{d+\lambda+1}\right),\nonumber
  \end{align}    
and the lower bound  in the statement of the lemma follows.
We now prove the upper bound.
If $t\le 3/4$, then $g(r)\ge 1-(3/4)^2$, hence $|f'(r)|\le c$ and $f(r)\le 1+cr$.
This and \autoref{eq:ru} yield $u_\lambda(t)-u_0(t) \leq \frac{1}{\alpha-\lambda}+C$, cf.~\autoref{eq:lbe}. 

We denote 
$s=\sqrt{2(1-t)}$,
so that 
$s\in(0,2]$.
If $\theta,\eta\in \Sd$, $t=\theta\cdot\eta$ and $\gamma$ is the angle between $\theta$ and $\eta$,
then $t=\cos \gamma$ and 
\begin{align}\label{eq:st}
s=2\sin(\gamma/2)=|\theta-\eta|.
\end{align}

We 
need to consider $t>3/4$, or 
$s^2<1/2$.
Let $x\in(0,1/2)$.
By \autoref{eq:ru},
  \begin{align*}
    u_\lambda(t)-u_0(t)
    &=
    \left(\int_0^x+\int_x^1\right) \left( 1-r^{\lambda} \right)\left( r^{\alpha-\lambda-1}-r^{d-1} \right)
    \left( (r-1)^2+rs^2\right)^{-\frac{d+\alpha}2}dr \\
    &= \mathrm{I}+{\rm I\!I}.
  \end{align*}
We have 
\begin{align*}
    {\rm I} &\le
  \int_{0}^{x} r^{\alpha-\lambda-1} (1-r)^{-d-\alpha}dr
  \le
  (1-x)^{-d-\alpha}\int_{0}^{1} r^{\alpha-\lambda-1} dr
  \\&=
  \frac{(1-x)^{-d-\alpha}}{\alpha-\lambda}
  \le  \frac{1+Cx}{\alpha-\lambda}.
\end{align*}
To estimate ${\rm I\!I}$, we denote $f_a^s(v)=(1+vs)^{a}$ and \[\Delta f_a^s(v)=\frac{f_a^0(v)-f_a^s(v)}s=\frac{1-f_a^s(v)}s,\]
where $a>0$, $s>0$ and $-1/s\le v\le 0$. 
For $0<y<1$ and $a>0$ we have
  \begin{align*}
    (1-y)^a\ge (1-y)^{a\vee 1}\ge 1-(a\vee1)y.
  \end{align*}
  Therefore $  (1-(1-y)^a)/y \le a\vee 1$.
  Putting $y=|v|s$ we get
  \begin{align}\label{eq:odf}
    \Delta f_a^s(v)&\le (a\vee1)|v|.
  \end{align}
Substituting $r-1=v s$ in the integral defining ${\rm I\!I}$, we get
  \begin{align*}
        {\rm I\!I} &\le x^{\alpha-\lambda-1}\int_{x}^1 \left( 1-r^{\lambda} \right)\left(1-r^{d-\alpha+\lambda} \right)\left( (r-1)^2+rs^2\right)^{-\frac{d+\alpha}2}dr
    \\&=x^{\alpha-\lambda-1}s^{-d-\alpha+1}\int_{(x-1)/s}^0 \left( 1-f_\lambda^s(v) \right)\left( 1-f_{d-\alpha+\lambda}^s(v) \right)\left( v^2+(1+vs)\right)^{-\frac{d+\alpha}2}dv
    \\&=x^{\alpha-\lambda-1}s^{-d-\alpha+3}\int_{(x-1)/s}^0 \Delta f_\lambda^s(v)\Delta f_{d-\alpha+\lambda}^s(v)\left( v^2+vs+1\right)^{-\frac{d+\alpha}2}dv.
  \end{align*}
  Note that $-1\le x-1\le vs\le 0$. By \autoref{eq:odf} we have
  \begin{align*}
    \Delta f_\lambda^s(v)&\le (\lambda\vee1)|v|\le 2|v|,\\
    \Delta f_{d-\alpha+\lambda}^s(v)&\le ( (d-\alpha+\lambda)\vee1)|v|\le d|v|.
  \end{align*}
 Recall that $s^2<1/2$. By the above and a change of variables it follows 	that
 \begin{align*}
     {\rm I\!I} &\le C s^{-d-\alpha+3}x^{\alpha-\lambda-1}\int_{(x-1)/s}^0   
      \frac{v^2}{(v^2+vs+1)^{(d+\alpha)/2}}
    \\&\le Cs^{-d-\alpha+3}x^{\alpha-\lambda-1}
    \left(\int_{-s}^0  \frac{v^2}{(v^2+1/2)^{(d+\alpha)/2}}dv
    +\int_{-1/s}^{-s}  \frac{(-v)\sqrt{v^2+vs+1}}{(v^2+vs+1)^{(d+\alpha)/2}}dv\right)
    \\&\le 
    C s^{-d-\alpha+3}x^{\alpha-\lambda-1}
    \left(\int_{-1}^0 \frac{v^2}{(v^2+1/2)^{(d+\alpha)/2}}dv
    +\int_{-1/s}^{-s} \frac{-(2v+s)}{(v^2+vs+1)^{(d+\alpha-1)/2}}dv\right)
    \\&\le 
    C s^{-d-\alpha+3}x^{\alpha-\lambda-1}
    \left(1+ \int_{1}^{1/s^2}  \frac{du}{u^{(d+\alpha-1)/2}}\right)
    \\&\le 
    Cx^{\alpha-\lambda-1}(1\vee s^{-d-\alpha+3}),
 \end{align*}
provided $d+\alpha\neq 3$, and ${\rm I\!I}\le -x^{\alpha-\lambda-1}\log s$ if  $d+\alpha= 3$.
Let $x=(\alpha-\lambda)^\delta$.
Since $(\alpha-\lambda)^{\delta(\alpha-\lambda)}\le 1$, the upper bound follows.
The proof of \autoref{bound3} is complete.

\subsection{Proof of \autoref{lem:intexit}}

We consider transformation 
$$
W(y)=\frac{2}{|y|^2}y-\one, \qquad y\in \Rdz.
$$
Note that $W^{-1}(y)=T\left(\frac12 (\one+y)\right)$, $y\neq -\one$.
For $y=(w,1)\in F$, we have 
$$W(y)=\left[\begin{array}{c}
2(1+|w|^2)^{-1}w\\
-1+2(1+|w|^2)^{-1}
\end{array}\right],\quad
\frac{\partial W}{\partial w}=\left[
\begin{array}{c}
2(1+|w|^{-1}I_{d-1}-4 (1+|w|^2)^{-2}ww^T\\
-4(1+|w|^2)^{-2}w^T
\end{array}
\right],
$$
where $w\in \R^{d-1}$ is considered as a column vector, $w^T$ is the transpose of $w$ and
$I_{d-1}$ is the $(d-1)\times (d-1)$ identity matrix. By \cite[Proposition~12.13]{MR2245472} the surface measure on $W(F)=\Sd\setminus\{-\one\}$ is given by
\begin{align}\label{eq:sJ}
\sigma(W(dw))=\left[\det \left (\frac{\partial W}{\partial w}^T\frac{\partial W}{\partial w}\right)\right]^{1/2}dw.
\end{align}
We have 
$$
\frac{\partial W}{\partial w}^T\frac{\partial W}{\partial w}=4(1+|w|^2)^{-2}I_{d-1}+16|w|^2(1+|w|^2)^{-4}ww^T.
$$
By the matrix determinant lemma (see, e.g, \cite[Corollary~18.1.3]{MR1467237}), 
\begin{align*}
&\det \left (\frac{\partial W}{\partial w}^T\frac{\partial W}{\partial w}\right) =
 \left(4(1+|w|^2)^{-2}\right)^{d-1}
\left(1+4|w|^4(1+|w|^2)^{-2}\right),
\end{align*}
therefore $\sigma(W(dw))=2^{d-1}[1+\BigO(\varepsilon^2)]dw$ if $|w|<\varepsilon$.
We have $W(F\cap \Pi_\varepsilon)=B_\Theta$ (see \autoref{eoT} and \autoref{fig:1}).
In fact $\phi(W(y))=2^\alpha s^*_\varepsilon(y)=2^\alpha s_\varepsilon(y)$ for $y\in F\cap \Pi_\varepsilon$, cf. \autoref{eq:idu}.
Thus,  
\begin{align*}
    \int_{B_\Theta} \phi(\theta)\sigma(d\theta)&=
    2^{d-1+\alpha}\left[1+\BigO(\varepsilon^2)\right]\int_{\R^{d-1}\cap \{|w|<\varepsilon\}} \cda(\varepsilon^2-|w|^2)^{\alpha/2} dw
 \\&=2^{d-1+\alpha}{\omega_{d-1}}\cda \left[1+\BigO(\varepsilon^2)\right]\int_0^\varepsilon  (\varepsilon^2-r^2)^{\alpha/2}r^{d-2}dr
    \\&=2^{d-1+\alpha}\omega_{d-1}\cda \varepsilon^{\alpha+d-1}\left[1+\BigO(\varepsilon^2)\right]\int_0^1 (1-r^2)^{\alpha/2}r^{d-2}dr
    \\&=2^{d+\alpha-2}\omega_{d-1}\cda B\left(1+\frac{\alpha}{2},\frac{d-1}{2}\right)\varepsilon^{\alpha+d-1}\left[1+\BigO(\varepsilon^2)\right].
  \end{align*}
We finish the proof of \autoref{lem:intexit}
by recalling that $\varepsilon=\Theta/2+\BigO(\Theta^2)$, cf. \eqref{eoT}.

\subsection{Proof of \autoref{lem:Rlgd}}
Recall that $0<\varepsilon<1/20$
and $\Theta\le\pi/30$, in particular 
$|\theta-\eta|\le 1$
if
$\theta\in B_\Theta$, cf. \autoref{eq:st}.
By \autoref{eq:rul}, the definition of $\phi$, \autoref{bound3} and \autoref{eoT} we have
\begin{align*}
&\Ad\frac{1-c(\alpha-\beta)}{\alpha-\lambda}\int_{B_\Theta}\phi(\theta)\sigma(d\theta)\le 
    R_\lambda\phi(\eta)
    \\&\le 
    \Ad\frac{1+C(\alpha-\lambda)^\delta}{\alpha-\lambda}\int_{B_\Theta}\phi(\theta)\sigma(d\theta)
    +C \frac{\Theta^\alpha}{(\alpha-\lambda)^{\delta}}
        \int_{B_\Theta} \left(1\vee |\theta-\eta|^{-(d+\alpha-3)}\right) \sigma(d\theta).
    \end{align*}
The lower bound in \autoref{lem:Rlgd} follows immediately from \autoref{lem:intexit}.

To prove the upper bound  we first assume that
$d+\alpha\le 3$. Then $d=2$, $\alpha\le 1$ and $ \Theta^{\alpha}\le\omega(\Theta)$. 
Since 
$\displaystyle\int_{B_\Theta}  \sigma(d\theta) \asymp \Theta^{d-1}=\Theta$, 
by \autoref{lem:intexit}
we obtain, 
\[ R_\lambda\phi(\eta) \le 
    \Ad \widetilde C_{d,\alpha}\Theta^{d-1+\alpha}\left[1+\BigO(\Theta^2)\right]  
    \frac{1+C(\alpha-\lambda)^\delta}{\alpha-\lambda} 
    +C \frac{\Theta^{1+\alpha}}{(\alpha-\lambda)^{\delta}},
 \]
as needed.
We now  assume that $d+\alpha>3$.
It is not difficult to see that
\begin{align*}
\int_{B_\Theta}  |\eta-\theta|^{-(d+\alpha-3)} \sigma(d\theta) \asymp \Theta^{2-\alpha}.
\end{align*}
Consequently, 
\[ R_\lambda\phi(\eta) \le    
   \Ad \widetilde C_{d,\alpha}\Theta^{d-1+\alpha}\left[1+\BigO(\Theta^2)\right]
        \frac{1+C(\alpha-\lambda)^\delta}{\alpha-\lambda}
        +C \frac{\Theta^{2}}{(\alpha-\lambda)^{\delta}}.
\]        
But $\Theta\le \omega(\Theta)$, which yields \autoref{radialupper} in this case, too.
The proof of \autoref{lem:Rlgd} 
is complete. In fact we proved a stronger estimate for $\alpha=1$.

\subsection{Proof of \autoref{cor:slit}}
The following is a folklore connection between harmonic functions of $\Delta^{1/2}$ and $\Delta$.
\begin{lemma}\label{lem:he}
Let $d\in \{1,2,\ldots\}$ and
$$
P_t(x)=\frac{2}{\omega_{d+1}}\frac{t}{(|x|^2+t^2)^{(d+1)/2}}, \qquad t>0, \;x\in \R^d.
$$
If function $\Phi$ on $\R^n$ is harmonic for $\Delta^{1/2}$ in a open set $E\subset \R^d$, 
and for $x\in \Rd$ we let
\begin{align*}
U(x,t)=\left\{
\begin{array}{lll}
P_t*\Phi(x), & \mbox{ if} & t>0 ,\\
\Phi(x),& \mbox{ if} & t=0 ,\\
P_{-t}*\Phi(x),& \mbox{ if} & t<0 ,
\end{array}\right.
\end{align*}
then $U$ is harmonic for $\Delta$ in $D=\{(x,t)\in \R^{d+1}:t\neq 0 \mbox{ or } x \in E\}$.
\end{lemma}
\begin{proof}
$U$ is well-defined because of \eqref{eq:dc}. It is harmonic (for $\Delta$ in $d+1$ variables) on $\R^{d+1}\setminus \{t=0\}$ and continuous on $D$, cf.  \cite[Chapter III]{MR0290095}.
It is well-known and easy to derive directly  
that $\partial U(x,t)/\partial t=\Delta^{1/2}\Phi(x)$ at $t=0$ and $x\in E$ (hint:  $\int_{\Rd}P_t(y)dy=1$). 
Since $\Phi$ is $1/2$-harmonic on $E$, the derivative equals zero at $t=0$ and $x\in E$.
It follows that $V(x,t)=\partial U(x,t)/\partial t$ is continuous in $D$.
By the reflection principle for harmonic functions, $V$ is harmonic in $D$. For $x\in D$ and $t\in \R$ we have
$U(x,t)=U(x,1)+\int_1^t V(x,s)ds$. Thus $U$ is $C^2$ in $D$, and so $\Delta U=0$ on $D$.
\end{proof}
Let  $M$ be the Martin kernel of the cone $\Gamma_\Theta\subset \R^d$ for $\Delta^{1/2}$ and $d\ge 2$.
The above harmonic extension of $M$ to $\R^{d+1}$ is a constant multiple of the Martin kernel (with the pole at infinity) for $V$ and $\Delta$. Indeed, \cite[Corollary 1]{MR608327} asserts that all nonnegative harmonic functions vanishing at $E$ are proportional, see also \cite[Theorem~1.1]{MR2337478}. By Theorem~\ref{mainthm} and a change of variables, $M*P_{kt}(kx)=k^\beta M*P_t(x)$, where $k>0$. We have $\beta=1-B_{d,1}\Theta^{d}+\BigO(\Theta^{d+1}\log \Theta)$ and $$B_{d,1}=\frac{1}{2\pi}\frac{d-1}{d}\frac{\Gamma(\frac{d-1}{2})^2}{\Gamma(\frac{d}{2})^2}.$$ 
Since $B_{2,1}=1/4$,
\autoref{cor:slit} follows. In fact we proved a more general result.

\end{document}